\documentclass[a4paper,12pt,reqno]{article}

\usepackage{amsthm}
\usepackage{amsmath,amssymb,latexsym,amsfonts,mathrsfs}
\usepackage[dvipdfmx]{graphicx}
\usepackage{bm}

\usepackage{color}
\usepackage{comment}
\usepackage{mathtools}
\usepackage{fancyhdr}
\usepackage[top=30truemm,bottom=30truemm,left=20truemm,right=20truemm]{geometry}
\usepackage{enumitem}
\usepackage{enumerate}

\newcommand{\R}{\mathbb{R}}
\newcommand{\C}{\mathbb{C}}

\newcommand{\e}{\varepsilon}

\newcommand{\SCR}[1]{{\mathscr #1}}

\newcommand{\CAL}[1]{{\cal #1}}

\newcommand{\J}[1]{\left\langle #1 \right\rangle}

\theoremstyle{plain}
\newtheorem{Thm}{{\bf Theorem}}[section]

\newtheorem{Lem}[Thm]{{\bf Lemma}}
\newtheorem{Prop}[Thm]{{\bf Proposition}}

\theoremstyle{definition}
\newtheorem{Def}[Thm]{{\bf Definition}}
\newtheorem{Rem}[Thm]{{\bf Remark}}

\newcounter{Exami}

\allowdisplaybreaks[2]

\makeatletter
 \def\address#1#2{\begingroup
 \noindent\parbox[t]{7.8cm}{%
 \small{\scshape\ignorespaces#1}\par\vskip0ex
 \noindent\small{\itshape E-mail}%
 \/: #2\par\vskip4ex}\hfill%
 \endgroup}%
\makeatother

\begin{document}
\fontencoding{T1}\selectfont
%%%%%%%%%%%%%%%%%%%%%%%%%%%%%%%%%%%%%%%%%%%%%%%%%%%%%%%%%%%%%%% 
\begin{comment}
\end{comment}
%%%%%%%%%%%%%%%%%%%%%%%%%%%%%%%%%%%%%%%%%%%%%%%%%%%%%%%%%%%%%%%  
\title{Sharp asymptotic behavior of solutions to damped nonlinear Schr\"{o}dinger equations}
\author{Kodai Takagi, Shun Takizawa}
\date{\today}
\maketitle
\begin{abstract}
We consider large time asymptotics for damped nonlinear Schr\"{o}dinger equations. It is known 
that the nonlinear solution asymptotically behaves like a linear solution when time $t$ tends to infinity in the energy space. 
We prove that its convergence rate can be refined and the obtained rate is sharp if initial data belong to certain function spaces. This result partially solves open problems concerning the optimal decay rate of scattering.
\end{abstract}
%%%%%%%%%%%%%%%%%%%%%%%%%%%%%%%%%%%%%%%%%%%%%%%%%%%%%%%%%%%%%%%%%%%%%%%%%%%%%%%%%%%%%
%%%%%%%%%%%%%%%%%%%%%%%%%%     セクション１（イントロ）      %%%%%%%%%%%%%%%%%%%%%%%%%%%%%%%%%%%%%
%%%%%%%%%%%%%%%%%%%%%%%%%%%%%%%%%%%%%%%%%%%%%%%%%%%%%%%%%%%%%%%%%%%%%%%%%%%%%%%%%%
\section{Introduction} 
In this paper, we study the following Cauchy problem for the damped nonlinear Schr\"{o}dinger equation:
\begin{align}\label{CP1}
\begin{cases}
&i\partial_{t}u+\Delta u+ia u=\mu |u|^{p-1} u, \hspace{3mm}(t,x)\in [0,\infty) \times \R^n,
\\&u|_{t=0}=u_{0},
\end{cases}
\end{align}
where $a>0, \mu=\pm 1$ and $p>1$. The damped nonlinear Schr\"{o}dinger equation
appears in various areas of nonlinear optics, plasma physics, and fluid
dynamics, and has been studied in both mathematical and physical viewpoints. 
See \cite{Goldman, MTsutsumi1, MTsutsumi2, Fibich, Ohta-Todorova, Mohamad, Hamouda}.

There is a large amount of literature on the asymptotic behavior for (\ref{CP1}) without damping (i.e. $a=0$).
Regarding $L^2$-scattering we refer to \cite{Tsutsumi-Yajima, Ozawa, Barab, Hayashi, Kita} and references therein. Concerning $H^1$-scattering we refer to \cite{YTsutsumi, Ginibre-Velo, Nakanishi, Bourgain, Dodson, BGTV} and bibliographies therein.

On the other hand, the asymptotic behavior for (\ref{CP1}) was first studied by Inui \cite{Inui}.
He proved that if  there exists a $\phi\in H^1$ such that $\|e^{at}u(t)-e^{it\Delta}\phi\|_{H^1}\to0$ as $t\to\infty$  for the solution $u\in C([0,\infty); H^1)$ with $u_0\in H^1$, then its convergence rate is obtained as follows:
\begin{align} \label{Inq1}
\|u(t)-e^{-at}e^{it\Delta}\phi\|_{H^1}= 
\begin{cases}
&o\left(t^{(p-1)\alpha+\e}e^{-apt}\right) \hspace{5mm}\text{if}\hspace{3mm}1< p<1+\frac{4}{n},
\\
&o\left(e^{-apt}\right)\hspace{20mm}\text{if}\hspace{3mm}1+\frac{4}{n}\leq p<1+\frac{4}{n-2},
\end{cases}
\end{align} 
for any $\e>0$ and $\alpha=\frac{4-(n-2)(p-1)}{2(p-1)(p+1)}>\frac{n}{2(n+2)}$.
Subsequently Aloui-Jbari-Tayachi \cite{Aloui} improved this result. More precisely, they proved
\begin{align} \label{Inq2}
\|u(t)-e^{-at}e^{it\Delta}\phi\|_{H^1}=o\left(e^{-apt}\right)
\end{align} 
with $1<p<1+4/n$. The asymptotic estimates (\ref{Inq1}) and (\ref{Inq2}) imply that there is some room for additional decay.  

The aim of our study is to specify this additional decay rate precisely.
In the present paper, we prove that the convergence rate of scattering is refined and its rate is optimal if initial data belong to the modulation space $M^{1,1}$ or the weighted Sobolev space $\Sigma=H^{[n/2]+1}\cap \SCR{F}H^{[n/2]+1}$, whose  definitions are given in Section 2. 
Here $[x]$ denotes the greatest integer not greater than $x$. We first state the result on $L^2$-scattering.
\begin{Thm} \label{thm2}
Let $u_0\in M^{1,1}$ and the power $p$ be an odd integer. Suppose that the solution $u\in C([0,\infty); M^{1,1})$ of (\ref{CP1}) such that 
\begin{align} \label{Asmp1}
\|u(t)\|_{M^{1,1}}\leq C e^{-\frac{ap+\e}{2p-1}t}, \hspace{3mm}t>0
\end{align} 
with some $C>0$ and small $\e>0$.
Then, there exist unique $\phi \in M^{1,1}$ and constants $C_1, C_2>0$ and large $T>0$ such that 
\begin{align} \label{MainInq2}
C_{1} t^{-n(p-1)/2} e^{-apt} \leq \|u(t)-e^{-at}e^{it\Delta}\phi\|_{L^2}\leq C_{2} t^{-n(p-1)/2} e^{-apt}
\end{align}
for all $t>T$.
\end{Thm}
The space $M^{1,1}$ has low regularity in general. In other words, $M^{1,1}(\R^n) \not\subset H^1(\R^n)$ holds, for which a concrete example is given in Proposition \ref{Hougan}. To the best of our knowledge, Theorem \ref{thm2} is the first result on sharp asymptotics for (\ref{CP1}) with rough data both with and without damping. 

Next, we can deal with energy scattering by employing another space instead of $M^{1,1}$. 
\begin{Thm} \label{thm1}
Let $u_0\in \Sigma$ and  $p$ be an odd integer or $p>[n/2]+1$. Assume that the solution $u\in C([0,\infty); \Sigma)$ of (\ref{CP1}) such that 
\begin{align} \label{Asmp2}
\|e^{-it\Delta}u(t)\|_{\Sigma}\leq C e^{-at}, \hspace{3mm}t>0
\end{align} 
with some $C>0$.
Then, there exist unique $\phi \in \Sigma$ and constants $C_1, C_2>0$ and large $T>0$ such that 
\begin{align} \label{MainInq1}
C_{1} t^{-n(p-1)/2} e^{-apt} \leq \|u(t)-e^{-at}e^{it\Delta}\phi\|_{H^{1}}\leq C_{2} t^{-n(p-1)/2} e^{-apt}
\end{align}
for all $t>T$.
\end{Thm}

Theorems \ref{thm2}, \ref{thm1} partially solve the open problem (see \cite[p.772]{Inui}) regarding the optimality of convergence rate to scattering. 

\begin{Rem}
The linear solution corresponding to (\ref{CP1}) consists of the dissipative part and the dispersive part: $u(t)=e^{-at} e^{it\Delta} u_0$. The exponential decay $e^{-apt}$ appeared in (\ref{Inq1})-(\ref{MainInq2}) arises from the dissipative part. On the other hand, the additional decay $t^{-n(p-1)/2}$ appeared in (\ref{MainInq2}), (\ref{MainInq1}) comes from the dispersive part.
\end{Rem}
We remark on the assumptions in our results.
\begin{Rem}
The conditions on the power in Theorems \ref{thm2} and \ref{thm1} are required in order to use the algebra property of $M^{1,1}$ (see Proposition \ref{PropM}) and the Kato-Ponce type inequality (see Proposition \ref{PropKP}), respectively. Note that no restriction is imposed on the size of the initial data although the decay conditions for the solutions are imposed.
The assumptions (\ref{Asmp1}), (\ref{Asmp2}) are satisfied at least for small data (see Propositions \ref{SDGE1} and \ref{SDGE2}).
\end{Rem}

\subsection*{Notation}
We write $F(u)=\mu|u|^{p-1}u$ by using $F: \C\to \C$ with $F(z)=\mu|z|^{p-1}z$. 
Let $\J{x}=(1+|x|^2)^{1/2}$.
We say that $u$ is a solution to (\ref{CP1}) if it satisfies the corresponding integral equation:
\begin{align*}
u(t)=e^{-at}e^{it\Delta}u_0-i\int_{0}^{t}e^{-a(t-s)} e^{i(t-s)\Delta} F(u(s))ds.
\end{align*} 
Such a solution is called for a mild solution.
We write
\begin{align*}
&\SCR{F}f(\xi)=\widehat{f}(\xi)=(2\pi)^{-n/2}\int_{\R^n}f(x)e^{-ix\cdot\xi}dx,
\\
&\SCR{F}^{-1}f(x)=\check{f}(x)=(2\pi)^{-n/2} \int_{\R^n}f(\xi)e^{ix\cdot\xi}d\xi
\end{align*} 
for the Fourier transform and the inverse Fourier transform of $f$, respectively. 
We define the free Schr\"{o}dinger propagator $e^{it\Delta}$ by
\begin{align*}
e^{it\Delta}f(x)=\SCR{F}^{-1} e^{-it|\xi|^2} \SCR{F}f(x)=(4\pi it)^{-n/2}\int_{\R^n} e^{i\frac{(x-y)^2}{4t}} f(y) dy.
\end{align*} 
We often use the notation $X\lesssim Y$ in the proofs if $X\leq CY$ with some constant $C>0$ depending only on the space dimension $n$, the damping parameter $a$ and the power $p$ of the nonlinearity.
\subsection*{Strategy of the proofs}
A key tool of our proofs is the following MDFM (or Dollard) decomposition formula:
\begin{align*}
e^{it\Delta}=M(t) D(t) \SCR{F} M(t),
\end{align*}
where $M(t)=e^{i|x|^2/(4t)}$ is the multiplier and $D(t)$ is the dilation which leaves the $L^2$-norm invariant by $D(t)f(x)=(2it)^{-n/2}f(\frac{x}{2t})$. This tool plays an important role to study sharp asymptotics for NLS without damping (\cite{Hayashi-Kawahara, Kita, Kita-Shimomura, Kita-Ozawa, Kita-Wada}) and we are inspired by methods developed in these works.

 When the solution $u(t)$ scatters to a linear solution $e^{-at}e^{it\Delta}\phi$, we have the Duhamel formula
\begin{align*}
u(t)-e^{-at}e^{it\Delta}\phi=\int_{t}^{\infty} e^{-a(t-s)} e^{i(t-s)\Delta}F(u(s))ds
\end{align*}
and expect that an asymptotic profile of the right hand side for large $t$ is the form 
\begin{align*}
\text{R.H.S.} & \sim e^{-at}\int_{t}^{\infty} e^{as} e^{i(t-s)\Delta} F(e^{-as}e^{is\Delta}\phi) ds
\\
& \sim e^{-at} \int_{t}^{\infty} e^{as} \left(M(t)D(t)\SCR{F}\right) \left(\SCR{F}^{-1} D(s^{-1}) M(-s)\right) F\left(e^{-as}M(s)D(s)\SCR{F}\phi\right)ds
\end{align*}
since $u(t)\sim e^{-at}e^{it\Delta}\phi$ and $e^{it\Delta}\phi=M(t)D(t)\SCR{F}M(t)\phi\sim M(t)D(t)\SCR{F}\phi$
 as $t\to \infty$. The decay rate $t^{-n(p-1)/2} e^{-apt}$ appeared in Theorems \ref{thm2} and \ref{thm1} is obtained by directly computing the last form.

%%%%%%%%%%%%%%%%%%%%%%%%%%%%%%%%%%%%%%%%%%%%%%%%%%%%%%%%%%%%%%%%%%%%%%%%%%%%%%%%%%%%%
%%%%%%%%%%%%%%%%%%%%%%%%%%%%%%%%%%%%%%%%%%%%%%%%%%%%%%%%%%%%%%%%%%%%%%%%%%%%%%%%%%%%%%%%%%%%%%%%%%%%%%%%%%%%%%%       $2 予備概念           %%%%%%%%%%%%%%%%%%%%%%%%%%%%
%%%%%%%%%%%%%%%%%%%%%%%%%%%%%%%%%%%%%%%%%%%%%%%%%%%%%%%%%%%%%%%%%%%%%%%%%%%%%%%%%%%%%%%%%%%%%%%%%%%%%%%%%%%%%%%%%%%%%%%%%%%%%%%%%%%%%%%%%%%%%%%%%%%%%%%%%%%%%%%%%%%%%%%%%%
%%%%%%%%%%%%%%%%%%%%%%%%%%%%%%%%%%%%%%%%%%%%%%%%%%%%%%%%%%%%%%%%%%%%%%%%%%%%%%%%%%%%%
\section{Function spaces}

We denote $\CAL{S}(\R^n)$ by the Schwartz space and $\CAL{S}'(\R^n)$ by the set of all the tempered distributions. 
Let $H^s$ be the usual $L^2$-based Sobolev space. The function space $\SCR{F}H^s$ is the weighted $L^2$ space, which is defined by 
\begin{align*}
\SCR{F}H^s=\{ f\in \CAL{S}'(\R^n) : \|f\|_{\SCR{F}H^s}=\|\J{\cdot}^{s} f\|_{L^2}<\infty \}.    
\end{align*}
We define $\Sigma=H^{[n/2]+1}\cap \SCR{F}H^{[n/2]+1}$ used in Theorem \ref{thm1}.
The following statement is the same as \cite[Proposition 2.2]{Kita}.
\begin{Prop} \label{PropKP}
Let $p$ be odd or $p>[n/2]+1$. Then it holds that 
\begin{align*}
\|F(u)\|_{H^{[n/2]+1}}\leq C \|u\|_{L^{\infty}}^{p-1} \|u\|_{H^{[n/2]+1}}
\end{align*}
for $u\in H^{[n/2]+1}$.
\end{Prop}

We introduce the definition of one of the modulation spaces also known as the Feichtinger algebra, which is based on the work of Feichtinger \cite{Fei1, Fei2, Fei3} and Gr\"{o}chenig \cite{Grochenig}.
\begin{Def}\label{DefM}
 Let $g\in \CAL{S}(\R^n)\setminus \{0\}$. Then we define
\begin{align*}
&M^{1,1}(\R^n)=\left\{f\in\CAL{S}'(\R^n) : \|f\|_{M^{1,1}}= \iint_{\R^{2n}} |V_{g}f(x,\xi)| dxd\xi<\infty\right\},
\end{align*}
where $V_{g}f(x,\xi)=\SCR{F} [\overline{g(\cdot-x)}f(\cdot)](\xi)$ is the short-time Fourier transform (STFT).
\end{Def}
It is known that the space $M^{1,1}$ does not depend on the function $g$.  
Modulation spaces have been used in the study of well-posedness and scattering for nonlinear dispersive equations (see e.g. \cite{Wang, Benyi-Okoudjou1, Sugimoto, TKato}).

We collect properties of the space $M^{1,1}(\R^n)$. We often write $M^{1,1}=M^{1,1}(\R^n)$ for short.
\begin{Prop} \label{PropM}
It holds that 
\begin{itemize}
\item If $f, g$ are in $M^{1,1}$, then their pointwise multiplication $fg$ is in $M^{1,1}$, that is, $\|fg\|_{M^{1,1}}\lesssim \|f\|_{M^{1,1}} \|g\|_{M^{1,1}}$.
\item If $f$ is in $M^{1,1}$, then $\SCR{F}f$ is in $M^{1,1}$, namely, $\|\SCR{F}f\|_{M^{1,1}}\lesssim \|f\|_{M^{1,1}}$.
\item $M^{1,1}\hookrightarrow L^1 \cap L^{\infty}$, in other words, $\|f\|_{L^p}\lesssim \|f\|_{M^{1,1}}$ for all $1\leq p\leq \infty$.
\item $\|e^{it\Delta}f\|_{M^{1,1}}\lesssim \J{t}^{n/2} \|f\|_{M^{1,1}}$ for all $t\in\R$.
\end{itemize}
\end{Prop}
These facts are well known. We refer to \cite[Proposition 2.24]{Benyi-Okoudjou2} for the first two  statements, \cite[Proposition 1.7 (3)]{Toft} for the third, and \cite[Corollary 18]{Benyi} for the fourth.
Additionally the space $M^{1,1}$ has no condition on derivatives.
\begin{Prop} \label{Hougan}
The inclusion $M^{1,1}(\R^n)\subset H^1(\R^n)$ does not hold.
\end{Prop}
This proof is postponed to Appendix for reader's convenience.
%%%%%%%%%%%%%%%%%%%%%%%%%%%%%%%%%%%%%%%%%%%%%%%%%%%%%%%%%%%%%%%%%%%%%%%%%%%%%%%%%%%%%%%%%
%%%%%%%%%%%%%%%%%%%%%%%%%%%%%%%%%%%%%%%%%%%%%%%%%%%%%%%%%%%%%%%%%%%%%%%%%%%%%%%%%%%%%%%%%%
%%%%%%%%%%%%%%%%%%%%%%%%%%%%%%%%%%%%%%%%%%%%%%%%%%%%%%%%%%%%%%%%%%%%%%%%%%%%%%%%%%%%%%%%%%
%%%%%%%%%%%%%%%%%%%%%%%%%%%%%%%%%%%%%%%%%%%%%%%%%%%%%%%%%%%%%%%%%%%%%%%%%%%%%%%%%%%%%%%%%%
%%%%%%%%%%%%%%%%%%%%%%%%%%%%%%%%%%%%%%%%%%%%%%%%%%%%%%%%%%%%%%%%%%%%%%%%%%%%%%%%%%%%%%%%%%
%%%%%%%%%%%%%%%%%%%%%%%%%%%%%%%%%%%%%%%%%%%%%%%%%%%%%%%%%%%%%%%%%%%%%%%%%%%%%%%%%%%%%%%%%%
%%%%%%%%%%%%%%%%%%%%%%%%%%%%%%%%%%%%%%%%%%%%%%%%%%%%%%%%%%%%%%%%%%%%%%%%%%%%%%%%%%%%%%%%%%
%%%%%%%%%%%%%%%%%%%%%%%%%%%%%%%%%%%%%%%%%%%%%%%%%%%%%%%%%%%%%%%%%%%%%%%%%%%%%%%%%%%%%%%%%%
%%%%%%%%%%%%%%%%%%%%%%%%%%%%%%%%%%%%%%%%%%%%%%%%%%%%%%%%%%%%%%%%%%%%%%%%%%%%%%%%%%%%%%%%%%
%%%%%%%%%%%%%%%%%%%%%%%%%%%%%%%%%%%%%%%%%%%%%%%%%%%%%%%%%%%%%%%%%%%%%%%%%%%%%%%%%%%%%%%%%%
\section{Small data global existence (SDGE)}
In this section, we prove small data global existence (hereafter abbreviated as SDGE) for (\ref{CP1}) in $M^{1,1}$ and $\Sigma$.

\begin{Prop} \label{SDGE1}
Let $p$ be odd and $u_{0}\in M^{1,1}$ be such that $\|u_{0}\|_{M^{1,1}}$ is sufficiently small.  Then there exists a unique solution $u \in C([0, \infty);M^{1,1})$ of (1). Moreover, the solution $u$ satisfies
\begin{align*}
\|u(t)\|_{M^{1,1}}\leq C \J{t}^{n/2} e^{-at}
\end{align*}
for all $t>0$. In particular, this solution $u$ fulfills (\ref{Asmp1}).
\end{Prop}
\begin{proof}
Let $v(t)=e^{at}u(t)$. Then the Cauchy problem (\ref{CP1}) is equivalent to 
\begin{align} \label{IntEqv}
\begin{cases}
&i\partial_{t}v+\Delta v=\mu e^{-a(p-1)t} |v|^{p-1} v,
\\
&v|_{t=0}=u_0.
\end{cases}
\end{align} 

We prove SDGE for (\ref{IntEqv}) by the standard contraction. 
Let us set $X=\{ v\in C([0,\infty); M^{1,1}) : \|v\|_{X}=\sup_{[0,\infty)} \J{t}^{-n/2} \|v(t)\|_{M^{1,1}}<\infty\}$, which is the Banach space. Applying the Duhamel formula to (\ref{IntEqv}), we have
\begin{align} \label{inteq}
 v(t)=e^{it\Delta}u_0-i \int_{0}^{t} e^{-a(p-1)s} e^{i(t-s)\Delta}F(v(s)) ds \equiv \Phi (v).
\end{align}
Taking $\J{t}^{-n/2} \|\cdot \|_{M^{1,1}}$ to the both sides of (\ref{inteq}) and using Proposition \ref{PropM}, we have
\begin{align*}
\J{t}^{-n/2}\|\Phi(v)(t)\|_{M^{1,1}} &\leq \J{t}^{-n/2} \|e^{it\Delta} u_0\|_{M^{1,1}}+\J{t}^{-n/2} \int_{0}^{t} e^{-a(p-1)s} \|e^{i(t-s)\Delta}F(v(s))\|_{M^{1,1}} ds
\\
&\lesssim \|u_0\|_{M^{1,1}}+ \int_{0}^{t} e^{-a(p-1)s} \|v(s)\|_{M^{1,1}}^{p} ds
\\
&\lesssim \|u_0\|_{M^{1,1}}+ \|v\|_{X}^{p} \int_{0}^{t} \J{s}^{np/2} e^{-a(p-1)s}  ds
\\
&\lesssim \|u_0\|_{M^{1,1}}+ \|v\|_{X}^{p}, 
\end{align*}
so that we have
\begin{align} \label{Inq211}
\|\Phi(v)\|_{X}\leq C_1 \|u_0\|_{M^{1,1}}+C_1 \|v\|_{X}^{p}
\end{align}
with some constant $C_1>0$ depending only on $a,n,p$.
Moreover, if $\|v\|_{X}, \|w\|_{X} \leq \delta$, then we have
\begin{align*}
& \hspace{5mm} \J{t}^{-n/2}  \|\Phi(v)(t)-\Phi(w)(t)\|_{M^{1,1}} 
  \\ 
 & \lesssim \int_{0}^{t}e^{-(p-1)as} \||v(s)|^{p-1}v(s)-|w(s)|^{p-1}w(s)\|_{M^{1,1}} \,ds 
\\
    & \lesssim \int_{0}^{t} e^{-(p-1)as} \left( \| v(s) \|_{M^{1,1}}^{p-1} + \| w(s) \|_{M^{1,1}}^{p-1} \right) \| v(s) - w(s) \|_{M^{1,1}} \,ds 
\\
    & \lesssim \delta^{p-1}\|v-w\|_{X}\int_{0}^{t}e^{-(p-1)as}\langle s \rangle^{\frac{np}{2}} \,ds,
\end{align*}
which implies 
\begin{align} \label{Inq212}
\|\Phi(v)-\Phi(w)\|_{X}\leq C_2 \delta^{p-1} \|v-w\|_{X}
\end{align}
for some constant $C_2>0$ relying only on $a,n,p$.
Hence, by choosing $\delta>0$ with $\delta<\min\{C_{1}^{1/(p-1)}, C_{2}^{1/(p-1)}\}$ and taking the initial data such as $\|u_0\|_{M^{1,1}}<(\delta-C_1 \delta^{p})/C_1$, we see that $\Phi$ is a contraction mapping on the closed ball of radius $\delta$ in $X$ by (\ref{Inq211}), (\ref{Inq212}). 
Thus we have a unique solution $v\in C([0,\infty); M^{1,1})$ since the solution $v\in X$ is continuous on $[0,\infty)$ in $M^{1,1}$ by a similar argument as above. Noting that $u(t)=e^{-at}v(t)$, we obtain the desired result. 
\end{proof}

Next we prove SDGE in $\Sigma$.
The following statement is the same as in \cite[Proposition 2.4]{Kita} if $a=0$.
\begin{Prop} \label{PropSol}
 Let $p$ be the same as in Theorem \ref{thm1} and $v_0 \in \Sigma$ with $\|v_0\|_{\Sigma} \le \delta$ sufficiently small. Then, there exists a unique solution $v(t)$ to (\ref{inteq}) such that
\begin{align*}
&v \in C([0,\infty); H^{[n/2]+1}),
\\
&\|v(t)\|_{H^{[n/2]+1}} + \sum_{|\alpha|=[n/2]+1} \|J(t)^{\alpha}v(t)\|_{L^2} \le 2\delta,
\\
&\|v(t)\|_{L^{\infty}} \le C\delta \langle t \rangle^{-n/2}.
\end{align*}
\end{Prop} 
The proof is almost the same argument as in \cite{Kita}. The only difference is the presence or absence of the exponential decay $e^{-a(p-1)s}$ in (\ref{inteq}), which is harmless.

\begin{Prop}\label{SDGE2}
Let $p$ be the same as in Theorem \ref{thm1} and  $u_0 \in \Sigma$ be such that $\|u_0\|_\Sigma$ is small enough. Then there exists a unique solution $u \in C([0, \infty); \Sigma)$ to (1). Moreover the solution $u$ fulfills
\begin{align*}
\|e^{-it\Delta}u(t)\|_{\Sigma}\leq C e^{-at}
\end{align*}
for all $t>0$.
\end{Prop}
\begin{proof} 
We consider (\ref{inteq}) by putting $v(t)=e^{at}u(t)$. By Proposition \ref{PropSol} there exists a solution $u \in C([0,\infty); H^{[n/2]+1})$ such that
\begin{align}
&\|u(t)\|_{H^{[n/2]+1}} + \sum_{|\alpha|=[n/2]+1} \|J(t)^{\alpha}u(t)\|_{L^2} \leq C e^{-at}, \label{inq101}
\\
&\|u(t)\|_{L^{\infty}} \leq C \langle t \rangle^{-n/2} e^{-at}  \label{inq102}
\end{align}
with some constant $C>0$ depending on $\delta$. The conservation law of $e^{it\Delta}$ and (\ref{inq101}) derive
\begin{align*}
\|e^{-it\Delta}u(t)\|_{H^{[n/2]+1}}=\|u(t)\|_{H^{[n/2]+1}} \lesssim e^{-at}
\end{align*}
and
\begin{align*}
\|e^{-it\Delta}u(t)\|_{\SCR{F}H^{[n/2]+1}}&=\|\J{x}^{[n/2]+1} e^{-it\Delta}u(t)\|_{L^2}
\\
&\lesssim \|e^{-it\Delta}u(t)\|_{L^2}+\sum_{|\alpha|=[n/2]+1} \|x^{\alpha} e^{-it\Delta}u(t)\|_{L^2}
\\
&= \|u(t)\|_{L^2}+\sum_{|\alpha|=[n/2]+1} \|J(t)^{\alpha} u(t)\|_{L^2}
\\
&\lesssim e^{-at},
\end{align*}
which complete the proof.
\end{proof}
%%%%%%%%%%%%%%%%%%%%%%%%%%%%%%%%%%%%%%%%%%%%%%%%%%%%%%%%%%%%%%%%%%%%%%%%%%%%%%%%%%%%%%%%%
%%%%%%%%%%%%%%%%%%%%%%%%%%%%%%%%%%%%%%%%%%%%%%%%%%%%%%%%%%%%%%%%%%%%%%%%%%%%%%%%%%%%%%%%%%
%%%%%%%%%%%%%%%%%%%%%%%%%%%%%%%%%%%%%%%%%%%%%%%%%%%%%%%%%%%%%%%%%%%%%%%%%%%%%%%%%%%%%%%%%%
%%%%%%%%%%%%%%%%%%%%%%%%%%%%%%%%%%%%%%%%%%%%%%%%%%%%%%%%%%%%%%%%%%%%%%%%%%%%%%%%%%%%%%%%%%
%%%%%%%%%%%%%%%%%%%%%%%%%%%%%%%%%%%%%%%%%%%%%%%%%%%%%%%%%%%%%%%%%%%%%%%%%%%%%%%%%%%%%%%%%%
%%%%%%%%%%%%%%%%%%%%%%%%%%%%%%%%%%%%%%%%%%%%%%%%%%%%%%%%%%%%%%%%%%%%%%%%%%%%%%%%%%%%%%%%%%
%%%%%%%%%%%%%%%%%%%%%%%%%%%%%%%%%%%%%%%%%%%%%%%%%%%%%%%%%%%%%%%%%%%%%%%%%%%%%%%%%%%%%%%%%%
%%%%%%%%%%%%%%%%%%%%%%%%%%%%%%%%%%%%%%%%%%%%%%%%%%%%%%%%%%%%%%%%%%%%%%%%%%%%%%%%%%%%%%%%%%
%%%%%%%%%%%%%%%%%%%%%%%%%%%%%%%%%%%%%%%%%%%%%%%%%%%%%%%%%%%%%%%%%%%%%%%%%%%%%%%%%%%%%%%%%%
%%%%%%%%%%%%%%%%%%%%%%%%%%%%%%%%%%%%%%%%%%%%%%%%%%%%%%%%%%%%%%%%%%%%%%%%%%%%%%%%%%%%%%%%%%
\section{Proof of Theorem \ref{thm2}}

The following Lemma is elementary estimates used in the subsequent proofs. 
\begin{Lem} \label{ElemLem}
Let $\alpha\ne0, \beta>0$. There exist constants $C_1,C_2>0$ and $T>0$ such that
\begin{align*}
C_1 t^{\alpha} e^{-\beta t}\leq \int_{t}^{\infty} s^{\alpha} e^{-\beta s} ds \leq C_2 t^{\alpha} e^{-\beta t}
\end{align*}
for all $t>T$.
\end{Lem}
\begin{proof}
This follows easily from integration by parts.
\end{proof}

We show the existence of a scattering state $\phi\in M^{1,1}$. 
\begin{Prop} \label{Cauchy1}
For the solution $u\in C([0,\infty); M^{1,1})$ to (\ref{CP1}) satisfying (\ref{Asmp1}), there exists a unique $\phi\in M^{1,1}$ such that
\begin{align*}
\lim_{t\to\infty}\|u(t)-e^{-at}e^{it\Delta}\phi \|_{M^{1,1}}=0.
\end{align*}
\end{Prop}
\begin{proof}
We prove that $\{ e^{at} e^{-it\Delta} u(t)\}_{t>0}$ is a Cauchy sequence in $M^{1,1}$. We obtain by Proposition \ref{PropM} and the assumption (\ref{Asmp1}) for $t>t'>T$,
\begin{align*}
\|e^{at} e^{-it\Delta} u(t)-e^{at'} e^{-it'\Delta} u(t')\|_{M^{1,1}}&=\left\| \int_{t'}^{t} e^{as} e^{-is\Delta} F(u(s)) ds \right \|_{M^{1,1}}
\\
&\lesssim  \int_{t'}^{t} e^{as} \J{s}^{n/2} \|u(s)\|_{M^{1,1}}^{p}ds
\\
&\lesssim  \int_{t'}^{t} \J{s}^{n/2} e^{as}  e^{-\frac{ap+\e}{2p-1}ps} ds\hspace{2mm}\to 0
\end{align*}
as $t'\to \infty$ because the last integrand belongs to $L^{1}([T,\infty))$.  
\end{proof}  

We decompose the Duhamel term of (\ref{CP1}) via the method of Kita \cite{Kita} in order to obtain the lower bound in (\ref{MainInq2}).
The factorization $e^{it\Delta}=M(t)D(t)\SCR{F} M(t)$ yields
\begin{align} \label{decomp1}
     v(t)-e^{it \Delta}\phi&= i \int_{t}^{\infty}e^{-a(p-1)s}e^{i(t-s)\Delta}F(v(s))\,ds \notag
\\
&= i \int_{t}^{\infty}e^{-a(p-1)s}e^{i(t-s)\Delta}F(e^{is\Delta}\phi)\,ds  \notag
\\
&\quad + i \int_{t}^{\infty}e^{-a(p-1)s}e^{i(t-s)\Delta}\{F(v(s))-F(e^{is\Delta}\phi)\} \,ds  \notag
\\
&=i \int_{t}^{\infty}e^{-a(p-1)s}e^{i(t-s)\Delta}[F(M(s)D(s)\widehat{\phi})]\,ds  \notag
\\
&\hspace{5mm}+i \int_{t}^{\infty}e^{-a(p-1)s}e^{i(t-s)\Delta}[F(e^{is\Delta}\phi)-F(M(s)D(s)\widehat{ \phi})]\,ds  \notag
\\
&\quad + i \int_{t}^{\infty}e^{-a(p-1)s}e^{i(t-s)\Delta}\{F(v(s))-F(e^{is\Delta}\phi)\} \,ds  \notag
\\
&\equiv I_{1}(t,x)+R_{1}(t,x)+R_{2}(t,x).
\end{align}
We further decompose $I_1(t,x)$ by using the gauge invariance of the nonlinearity (i.e. $F(M(s)u)=M(s)F(u)$) as follows:
\begin{align} \label{decomp2}
I_{1}(t,x) &=i \int_{t}^{\infty}e^{-a(p-1)s} s^{-n(p-1)/2} e^{i(t-s)\Delta}[M(s)D(s) F(\widehat{\phi})]\,ds  \notag
\\
&=i \int_{t}^{\infty}e^{-a(p-1)s} s^{-n(p-1)/2} M(t)D(t) \SCR{F} M(t) M(-s) \SCR{F}^{-1} F(\widehat{\phi})\,ds  \notag
\\
&=i  \int_{t}^{\infty}e^{-a(p-1)s} s^{-n(p-1)/2} M(t)D(t) F(\widehat{\phi})\,ds  \notag
\\
&\hspace{5mm}+i \int_{t}^{\infty}e^{-a(p-1)s} s^{-n(p-1)/2} M(t)D(t) \SCR{F} \{M(t)M(-s)-1\} \SCR{F}^{-1} F(\widehat{\phi})\,ds  \notag
\\
&\equiv I_2(t,x)+R_3(t,x).
\end{align}
We have
\begin{align*}
\|I_2(t)\|_{L^2}&=\int_{t}^{\infty} e^{-a(p-1)s} s^{-n(p-1)/2} ds \|F(\widehat{\phi})\|_{L^2},
\end{align*}
which implies with Lemma \ref{ElemLem} that
\begin{align} \label{I2}
A t^{-n(p-1)/2} e^{-a(p-1)t} \leq \|I_2(t)\|_{L^2} \leq B t^{-n(p-1)/2} e^{-a(p-1)t} 
\end{align}
for some constants $A, B>0$.
We have by the factorization $e^{it\Delta}=M(t)D(t)\SCR{F} M(t)$ and Proposition \ref{PropM},
\begin{align} \label{r1}
\|R_1(t)\|_{L^2}&\leq \int_{t}^{\infty} e^{-a(p-1)s} \| F(e^{is\Delta}\phi)-F(M(s)D(s)\widehat{ \phi})\|_{L^2} ds  \notag
\\
&\lesssim \int_{t}^{\infty} e^{-a(p-1)s} \left(\|e^{is\Delta}\phi\|_{L^{\infty}}^{p-1}+\| M(s)D(s)\widehat{ \phi}  \|_{L^{\infty}}^{p-1}\right) \|e^{is\Delta}\phi- M(s)D(s)\widehat{ \phi}\|_{L^2} ds  \notag
\\
&\lesssim \int_{t}^{\infty} e^{-a(p-1)s} s^{-n(p-1)/2} \left(\|\phi\|_{L^{1}}^{p-1}+\|\widehat{ \phi}  \|_{L^{\infty}}^{p-1}\right) \|\left( M(s)-1\right) \phi\|_{L^2} ds  \notag
\\
&\lesssim \| \phi\|_{M^{1,1}}^{p-1} \int_{t}^{\infty} e^{-a(p-1)s} s^{-n(p-1)/2}\|\left( M(s)-1\right) \phi\|_{L^2} ds . 
\end{align}
By Lemma \ref{ElemLem} and applying the fact $\displaystyle\lim_{s\to\infty}\|(M(s)-1)\phi\|_{L^2}=0$ to (\ref{r1}), there exists $T>0$ such that
\begin{align} \label{R1}
\|R_1(t)\|_{L^2}&\leq \frac{A}{3} t^{-n(p-1)/2} e^{-a(p-1)t}
\end{align}
for $t>T$.
We obtain by Proposition \ref{PropM} and (\ref{Asmp1})
\begin{align} \label{r2}
\|R_2(t)\|_{L^2}&\leq C \int_{t}^{\infty} e^{-a(p-1)s} \left(\| v(s) \|_{L^{\infty}}^{p-1}+ \|e^{is\Delta}\phi\|_{L^{\infty}}^{p-1}\right) \|v(s)-e^{is\Delta}\phi\|_{L^2} ds \notag
\\
&\leq C \left(1 + \|\phi\|_{M^{1,1}}^{p-1}\right)   \int_{t}^{\infty} e^{-a(p-1)s} \left(e^{a(p-1)s} e^{-\frac{ap+\e}{2p-1}(p-1)s} +s^{n(p-1)/2}\right) \|v(s)-e^{is\Delta}\phi\|_{L^2} ds  \notag
\\
&\leq C \int_{t}^{\infty} e^{-\frac{ap+\e}{2p-1}(p-1)s} \|v(s)-e^{is\Delta}\phi\|_{L^2} ds.
\end{align}
Here by (\ref{Asmp1}) we have 
\begin{align} 
\|v(s)-e^{is\Delta}\phi\|_{L^2}&\leq \int_{s}^{\infty} e^{-a(p-1)\tau} \|F(v(\tau))\|_{L^2} d\tau  \notag
\\
&\leq \int_{s}^{\infty} e^{-a(p-1)\tau} \|v(\tau)\|_{M^{1,1}}^{p} d\tau  \notag
\\
&\leq C  \int_{s}^{\infty} e^{-a(p-1)\tau} e^{ap\tau} e^{-\frac{ap+\e}{2p-1}p\tau} d\tau  \notag
\\
&\leq C e^{as-\frac{ap+\e}{2p-1}ps},  \notag
\end{align}
which implies with (\ref{r2}) that
\begin{align} \label{R2}
\|R_2(t)\|_{L^2}&\leq C \int_{t}^{\infty} e^{-\frac{ap+\e}{2p-1}(p-1)s} e^{as-\frac{ap+\e}{2p-1}ps} ds \notag
\\
&=C \int_{t}^{\infty} e^{-a(p-1)s-\e s} ds \notag
\\
&\leq C e^{-a(p-1)t} e^{-\e t}.
\end{align}
We obtain
\begin{align} \label{r31}
\|R_3(t)\|_{L^2}&\leq \int_{t}^{\infty} e^{-a(p-1)s} s^{-n(p-1)/2} \| (M(t)M(-s)-1) \SCR{F}^{-1} F(\widehat{\phi})\|_{L^2} ds.  
\end{align}
Since $\SCR{F}^{-1} F(\widehat{\phi})$ is in $L^2$ by Proposition \ref{PropM}, there exists $\varphi\in\CAL{S}$ such that $\| \SCR{F}^{-1} F(\widehat{\phi})-\varphi\|_{L^2}< A/(6 C_2)$ with $C_2$ being as in Lemma \ref{ElemLem} by a density, which yields that
\begin{align} \label{r32}
&\| (M(t)M(-s)-1) \SCR{F}^{-1} F(\widehat{\phi})\|_{L^2}  \notag
\\
&\leq \| (M(t)M(-s)-1) \varphi\|_{L^2}+\left\| \left(M(t)M(-s)-1\right) \left(\SCR{F}^{-1}  \notag F(\widehat{\phi})-\varphi\right)\right\|_{L^2}
\\
&\leq \left(\frac{1}{t}-\frac{1}{s}\right) \|\varphi\|_{\SCR{F}H^2}+2 \| \SCR{F}^{-1} F(\widehat{\phi})-\varphi\|_{L^2}  \notag
\\
&\leq t^{-1} \|\varphi\|_{\SCR{F}H^2}+A/(3 C_2).
\end{align}
Combining (\ref{r31}) and (\ref{r32}), we derive 
\begin{align} \label{R3}
\|R_3(t)\|_{L^2}&\leq \left( t^{-1} \|\varphi\|_{\SCR{F}H^2}+ \frac{A}{3 C_2} \right)\int_{t}^{\infty} e^{-a(p-1)s} s^{-n(p-1)/2}ds \notag
\\
&\leq C_2 \|\varphi\|_{\SCR{F}H^2} t^{-n(p-1)/2-1} e^{-a(p-1)t} + \frac{A}{3} t^{-n(p-1)/2} e^{-a(p-1)t}.
\end{align}
Therefore we attain by (\ref{decomp1}) and (\ref{decomp2})
\begin{align*}
\|I_2(t)\|_{L^2}-\sum_{j=1,2,3} \|R_j\|_{L^2} \leq \|v(t)-e^{it \Delta}\phi\|_{L^2}\leq \|I_2(t)\|_{L^2}+\sum_{j=1,2,3} \|R_j\|_{L^2},
\end{align*}
which implies with (\ref{I2}), (\ref{R1}), (\ref{R2}) and (\ref{R3}) that
\begin{align*}
C_1 t^{-n(p-1)/2}e^{-a(p-1)t} \leq \|v(t)-e^{it \Delta}\phi\|_{L^2}\leq C_2 t^{-n(p-1)/2}e^{-a(p-1)t}.
\end{align*}
This is the desired result.
%%%%%%%%%%%%%%%%%%%%%%%%%%%%%%%%%%%%%%%%%%%%%%%%%%%%%%%%%%%%%%%%%%%%%%%%%%%%%%%%%%%%%%%%%%
%%%%%%%%%%%%%%%%%%%%%%%%%%%%%%%%%%%%%%%%%%%%%%%%%%%%%%%%%%%%%%%%%%%%%%%%%%%%%%%%%%%%%%%%%%
%%%%%%%%%%%%%%%%%%%%%%%%%%%%%%%%%%%%%%%%%%%%%%%%%%%%%%%%%%%%%%%%%%%%%%%%%%%%%%%%%%%%%%%%%%
%%%%%%%%%%%%%%%%%%%%%%%%%%%%%%%%%%%%%%%%%%%%%%%%%%%%%%%%%%%%%%%%%%%%%%%%%%%%%%%%%%%%%%%%%%
%%%%%%%%%%%%%%%%%%%%%%%%%%%%%%%%%%%%%%%%%%%%%%%%%%%%%%%%%%%%%%%%%%%%%%%%%%%%%%%%%%%%%%%%%%
%%%%%%%%%%%%%%%%%%%%%%%%%%%%%%%%%%%%%%%%%%%%%%%%%%%%%%%%%%%%%%%%%%%%%%%%%%%%%%%%%%%%%%%%%%
%%%%%%%%%%%%%%%%%%%%%%%%%%%%%%%%%%%%%%%%%%%%%%%%%%%%%%%%%%%%%%%%%%%%%%%%%%%%%%%%%%%%%%%%%%
\section{Proof of Theorem \ref{thm1}}
We first prove the existence of a scattering state $\phi\in \Sigma$. 
\begin{Prop} \label{Cauchy2}
For the solution $u\in C([0,\infty); \Sigma)$ to (\ref{CP1}), there exists a unique $\phi\in \Sigma$ such that
\begin{align*}
\lim_{t\to\infty}\|u(t)-e^{-at}e^{it\Delta}\phi \|_{\Sigma}=0.
\end{align*}
\end{Prop}
The proof is analogous to that of Proposition \ref{Cauchy1} and \cite[Proposition 2.5]{Kita}. 

In order to prove (\ref{MainInq1}), we show the following estimates:
\begin{align}
&\|v(t)-e^{it\Delta}\phi\|_{H^1}\lesssim t^{-n(p-1)/2} e^{-a(p-1)t}, \label{Inq10}
\\
&\|v(t)-e^{it\Delta}\phi\|_{L^2}\gtrsim t^{-n(p-1)/2} e^{-a(p-1)t},  \label{Inq11}
\end{align}
where we note that $v(t)=e^{at}u(t)$.
We first estimate for (\ref{Inq10}).
We have by the dispersive estimate of $e^{it\Delta}$ and the Schwarz inequality
\begin{align*}
\|v(t)\|_{L^{\infty}}&=\|e^{it\Delta} e^{-it\Delta}v(t)\|_{L^{\infty}}
\\
&\lesssim |t|^{-n/2} \|e^{-it\Delta}v(t)\|_{L^1}
\\
&\lesssim |t|^{-n/2} \| \J{x}^{-[n/2]-1}\|_{L^2} \|\J{x}^{[n/2]+1}e^{-it\Delta}v(t)\|_{L^2}
\\
&\lesssim |t|^{-n/2}  \|e^{-it\Delta}v(t)\|_{\Sigma},
\end{align*}
which implies with the assumption (\ref{Asmp2}) that
\begin{align} \label{DispersiveEst}
\|v(t)\|_{L^{\infty}}\lesssim|t|^{-n/2}. 
\end{align}
The inequality (\ref{Inq10}) is obtained by Lemma \ref{ElemLem}, Proposition \ref{PropKP} and  (\ref{DispersiveEst}) as follows:
\begin{align*}
\|v(t)-e^{it\Delta}\phi\|_{H^1}&\leq \int_{t}^{\infty} e^{-a(p-1)s}\|F(v(s))\|_{H^{[n/2]+1}} ds
\\
&\lesssim \int_{t}^{\infty} e^{-a(p-1)s}\|v(s)\|_{L^\infty}^{p-1} \|v(s)\|_{H^{[n/2]+1}} ds
\\
&\lesssim \int_{t}^{\infty} e^{-a(p-1)s} s^{-n(p-1)/2} ds
\\
&\lesssim t^{-n(p-1)/2} e^{-a(p-1)t}.
\end{align*}
Second we show that (\ref{Inq11}) holds. For this purpose we employ the decomposition (\ref{decomp1}), (\ref{decomp2}) and prove the estimates (\ref{I2}), (\ref{R1}), (\ref{R2}) and (\ref{R3}) for $v(t), \phi\in \Sigma$. Replacing Proposition \ref{PropM} with Proposition \ref{PropKP}, we can show that these estimates also hold for $v(t), \phi\in \Sigma$.
Therefore we have (\ref{Inq11}), that is, the inequality (\ref{MainInq1}) follows.

\section*{Appendix}
In this section, we give the proof of Proposition \ref{Hougan}. 
\begin{proof}
We put $f(x)=\varphi(x)\sum_{k=1}^{\infty} k^{-3/2} e^{i k x_1}$, where $ \varphi \in \CAL{S}(\R^n) \setminus \{0\}$ with $\operatorname{supp} \hat{\varphi} \subset \{ \xi \in \R^n : |\xi| \le 1/4 \}$. 
It suffices to prove $f \in M^{1,1}(\R^n) \setminus H^{1}(\R^n)$.
First we show that $f \in M^{1,1}$. We have
\begin{align*}
    V_g f(x, \xi) &= \int \left( \sum_{k=1}^{\infty} \varphi(y) k^{-3/2} e^{i k y_1} \right) \overline{g(y - x)} e^{-i y \cdot \xi} \, dy 
\\
&=\sum_{k=1}^{\infty} k^{-3/2} V_{g}\varphi(x,\xi-ke_1),
\end{align*}
where $e_{1}=(1,0,\ldots,0)\in\R^n$.
Therefore, we obtain
\begin{align*}
    \|f\|_{M^{1,1}} &= \iint \left| \sum_{k=1}^{\infty} k^{-3/2} V_g \varphi(x, \xi - k e_1) \right| dxd\xi
\\
    &\le \sum_{k=1}^{\infty} k^{-3/2} \|\varphi\|_{M^{1,1}} \hspace{2mm}< \infty.
\end{align*}

Second we check that $f \notin H^1$ by contradiction.
Suppose that $f \in H^1$ and let $f_{N}(x):=\varphi(x) \sum_{k=1}^{N} k^{-3/2}e^{ikx_1}$. Then it follows from $\operatorname{supp} \hat{\varphi} \subset \{ \xi \in \R^n : |\xi| \le 1/4 \}$ that
\begin{align}
|\widehat{f_N}(\xi)|^2&=\left( \sum_{k=1}^{N} k^{-3/2} \widehat{\varphi}(\xi - k e_1) \right) \overline{\left( \sum_{j=1}^{N} j^{-3/2} \widehat{\varphi}(\xi - j e_1) \right)}  \notag
\\
&=\sum_{k=1}^{N} k^{-3} |\widehat{\varphi}(\xi - k e_1)|^2, \notag
\end{align}
 which yields
\begin{align} \label{eq100}
 \|\xi_1 \widehat{f_N}(\xi)\|_{L^2}^2
 &= \int_{\mathbb{R}^n} |\xi_1|^2 \sum_{k=1}^{N} k^{-3} |\widehat{\varphi}(\xi - k e_1)|^2 d\xi \notag
\\
&= \int_{\R^n}\xi_{1}^{2}|\widehat{\varphi}(\xi)|^2 d\xi  \sum_{k=1}^{N} k^{-3} +2 \int_{\R^n}\xi_1|\widehat{\varphi}(\xi)|^2 d\xi \sum_{k=1}^{N} k^{-2} +  \|\varphi\|_{L^2}^2  \sum_{k=1}^{N} k^{-1}.
\end{align}
The sequence of functions $\{\widehat{f_N}\}$ converges pointwise to $\widehat{f}$ and we have
 $\|\xi_1 \widehat{f_N}(\xi)\|_{L^2}^2 \to  \|\xi_1 \widehat{f}(\xi)\|_{L^2}^2$ as $N\to\infty$ by Lebesgue's monotone convergence theorem.
This limit is finite by the assumption of $f\in H^1$. On the other hand, right hand side of (\ref{eq100}) goes to infinity as $N\to \infty$, which is contradiction.
\end{proof}
%%%%%%%%%%%%%%%%%%%%%%%%%%%%%%%%%%%%%%%%%%%%%%%%%%%%%%%%%%%%%%%%%%%%%%%%%%%%%%%%%%%%
%%%%%%%%%%%%%%%%%%%%%%%%%%%%%%%%%%%%%%%%%%%%%%%%%%%%%%%%%%%%%%%%%%%%%%%%%%%%%%%%%%%%
%%%%%%%%%%%%%%%%%%%%%%%%%%%%%%%%%%%%%%%%%%%%%%%%%%%%%%%%%%%%%%%%%%%%%%%%%%%%%%%%%%%%%
%%%%%%%%%%%%%%%%%%%%%%%%%%%%%%%%  $6 参考文献  %%%%%%%%%%%%%%%%%%%%%%%%%%%%%%%%%%%%%%%%
%%%%%%%%%%%%%%%%%%%%%%%%%%%%%%%%%%%%%%%%%%%%%%%%%%%%%%%%%%%%%%%%%%%%%%%%%%%%%%%%%%%%%
%%%%%%%%%%%%%%%%%%%%%%%%%%%%%%%%%%%%%%%%%%%%%%%%%%%%%%%%%%%%%%%%%%%%%%%%%%%%%%%%%%%%%

\address{
Kodai Takagi\\
Department of Mathematics,\\
 Faculty of Science, \\
Tokyo University of Science,\\
Kagurazaka 1-3, Shinjuku-ku, \\
Tokyo 162-8601, Japan}
{1124517@ed.tus.ac.jp}

\address{
Shun Takizawa\\
Department of Mathematics,\\
 Faculty of Science, \\
Tokyo University of Science,\\
Kagurazaka 1-3, Shinjuku-ku, \\
Tokyo 162-8601, Japan}
{1123703@ed.tus.ac.jp}


\begin{thebibliography}{99} 
\bibitem{Aloui}{L.~Aloui, S.~Jbari and S.~Tayachi,}
\newblock{\em Asymptotic behavior and life-span estimates for the damped inhomogeneous nonlinear Schr\"{o}dinger equation,}
\newblock Evol. Equ. Control Theory,
{\bf 13} (2024), no.~4, 1126--1150.

\bibitem{Barab}{J. Barab,}
\newblock{\em Nonexistence of asymptotically free solutions for a nonlinear Schrodinger equation,}
\newblock J. Math. Phys.
{\bf 25} (1984), 3270--3273.

\bibitem{Benyi}{A.~B\'{e}nyi, K.~Gr\"{o}chenig, K.~A.~Okoudjou and L.~G.~Rogers,}
\newblock{\em Unimodular Fourier multipliers for modulation spaces,}
\newblock J. Funct. Anal.
{\bf 246} (2007), no.~2, 366--384.

\bibitem{Benyi-Okoudjou1}{A.~B\'{e}nyi and K.~A.~Okoudjou,}
\newblock{\em Local well-posedness of nonlinear dispersive equations on modulation spaces,}
\newblock  Bull. Lond. Math. Soc.
{\bf 41} (2009), 549--558.

\bibitem{Benyi-Okoudjou2}{A.~B\'{e}nyi and K.~A.~Okoudjou,}
\newblock{\em Modulation spaces---with applications to pseudodifferential operators and nonlinear Schr\"{o}dinger equations,}
\newblock Applied and Numerical Harmonic Analysis, Birkh\"{a}user/Springer, New York,
{\bf } (2020).

\bibitem{Bourgain}{J. Bourgain,}
\newblock{\em Global wellposedness of defocusing critical nonlinear Schr\"{o}dinger
equation in the radial case,}
\newblock J. Am. Math. Soc. 
{\bf 12} (1999), 145--171.

\bibitem{BGTV}{N. Burq, V. Georgiev, N. Tzvetkov, N. Visciglia,}
\newblock{\em $H^1$ scattering for mass-subcritical NLS with short-range nonlinearity and initial data in $\Sigma$,}
\newblock Ann. Henri Poincar\'{e}
{\bf 24} (2023), 1355--1376.

\bibitem{Dodson}{B. Dodson,}
\newblock{\em Global well-posedness and scattering for the defocusing, $L^2$-critical,
nonlinear Schr\"{o}dinger equation when $d=2$,}
\newblock Duke Math. J.
{\bf 165}  (2016), 3435--3516.

\bibitem{Fei1}{H.~G.~Feichtinger,}
\newblock{\em Banach spaces of distributions of Wiener's type and interpolation,}
\newblock Functional Analysis and Approximation, (Oberwolfach, 1980), 153--165,
{\bf 60} Birkhäuser Verlag, Basel-Boston, Mass., (1981).

\bibitem{Fei2}{H.~G.~Feichtinger,}
\newblock{\em On a new Segal algebra,}
\newblock Monatsh. Math.
{\bf 92} (1981), no.~4, 269--289.

\bibitem{Fei3}{H.~G.~Feichtinger,}
\newblock{\em Modulation spaces on locally compact abelian groups,}
\newblock Technical report, University of Vienna, (1983).

\bibitem{Fibich}{G. Fibich,}
\newblock{\em Self-focusing in the damped nonlinear Schr\"{o}dinger equation,}
\newblock SIAM J. Appl. Math.
{\bf 61} (2001), 1680--1705.

\bibitem{Ginibre-Velo}{J. Ginibre, G. Velo}
\newblock{\em Scattering theory in the energy space for a class of nonlinear
Schr\"{o}dinger equations,}
\newblock J. Math. Appl. Neuv. S\'{e}r
{\bf 64} (1985), 363--401.

\bibitem{Goldman}{M. V. Goldman, K. Rypdal, B. Hafizi,}
\newblock{\em Dimensionality and dissipation in Langmuir collapse,}
\newblock Phys. Fluids.
{\bf 23} (1980), 945--955.

\bibitem{Grochenig}{K.~Gr\"{o}chenig,}
\newblock{\em Foundations of Time-Frequency Analysis,}
\newblock Applied and Numerical Harmonic Analysis, Birkh\"{a}user Boston, Inc., Boston, MA,
{\bf } (2001).

\bibitem{Hamouda}{M.~Hamouda and M.~Majdoub,}
\newblock{\em Energy scattering for the unsteady damped nonlinear Schr\"{o}dinger equation,}
\newblock Mediterr. J. Math.
{\bf 22} (2025), no.~2, Paper No.~44, 17 pp.

\bibitem{Hayashi-Kawahara}{N.~Hayashi, Y.~Kawahara, and P.~I.~Naumkin,}
\newblock{\em Lower bounds of asymptotics in time of solutions to nonlinear Schr\"{o}dinger equations in 3D,}
\newblock Nonlinear Anal. 
{\bf 65} (2006), no.~7, 1394--1410.

\bibitem{Hayashi}{N.~Hayashi and P.~I.~Naumkin,}
\newblock{\em Asymptotics for large time of solutions to the nonlinear Schr\"{o}dinger and Hartree equations,}
\newblock  Amer. J. Math.
{\bf 120} (1998), no.~2, 369--389.

\bibitem{Inui}{T.~Inui,}
\newblock{\em Asymptotic behavior of the nonlinear damped Schr\"{o}dinger equation,}
\newblock Proc. Amer. Math. Soc.
{\bf 147} (2019), no.~2, 763--773.

\bibitem{TKato}{T. Kato,}
\newblock{\em The global Cauchy problems for the nonlinear dispersive equations on modulation spaces,}
\newblock J. Math. Anal. Appl.
{\bf 413} (2014), 821--840.

\bibitem{Kita}{N.~Kita,}
\newblock{\em Sharp $L^r$ asymptotics of the small solutions to the nonlinear Schr\"{o}dinger equations,}
\newblock Nonlinear Anal.
{\bf 52} (2003), no.~4, 1365--1377.

\bibitem{Kita-Ozawa}{N.~Kita and T.~Ozawa,}
\newblock{\em Sharp asymptotic behavior of solutions to nonlinear Schr\"{o}dinger equations with repulsive interactions,}
\newblock Commun. Contemp. Math. 
{\bf 7} (2005), no.~2, 167--176.

\bibitem{Kita-Shimomura}{N.~Kita and K.~Shimomura,}
\newblock{\em Large time behavior of solutions to Schr\"{o}dinger equations with a dissipative nonlinearity for arbitrarily large initial data,}
\newblock J. Math. Soc. Japan
{\bf 61} (2009), no.~1, 39--64.

\bibitem{Kita-Wada}{N.~Kita and T.~Wada,}
\newblock{\em Sharp asymptotic behavior of solutions to nonlinear Schr\"{o}dinger equations in one space dimension,}
\newblock Funkcial. Ekvac.
{\bf 45} (2002), no.~1, 53--69.

\bibitem{Mohamad}{D.~Mohamad,}
\newblock{\em Blow-up for the damped $L^2$-critical nonlinear Schr\"{o}dinger equation,}
\newblock Adv. Differential Equations
{\bf 17} (2012), no.~3--4, 337--367.

\bibitem{Nakanishi}{K. Nakanishi,}
\newblock{\em Energy scattering for nonlinear Klein--Gordon and Schr\"{o}dinger
equations in spatial dimensions 1 and 2,}
\newblock J. Funct. Anal. 
{\bf 169} (1999), 201--225.

\bibitem{Ohta-Todorova}{M.~Ohta and G.~Todorova,}
\newblock{\em Remarks on global existence and blowup for damped nonlinear Schr\"{o}dinger equations,}
\newblock Discrete Contin. Dyn. Syst.
{\bf 23} (2009), no.~4, 1313--1325.

\bibitem{Ozawa}{T. Ozawa,}
\newblock{\em Long range scattering for nonlinear Schrödinger equations in one space dimension,}
\newblock Comm. Math. Phys.
{\bf 139} (1991), no.~3, 479--493.

\bibitem{Sugimoto}{M.~Sugimoto,B.~Wang and R.~Zhang,}
\newblock{\em Local well-posedness for the Davey-Stewartson equation in a generalized Feichtinger algebra,}
\newblock J. Fourier Anal. Appl.
{\bf 21} (2015), no.~5, 1105--1129.

\bibitem{Toft}{J.~Toft,}
\newblock{\em Continuity properties for modulation spaces, with applications to pseudo-differential calculus.~I,}
\newblock J. Funct. Anal. 
{\bf 207} (2004), no.~2, 399--429.

\bibitem{MTsutsumi1}{M. Tsutsumi,}
\newblock{\em Nonexistence of global solutions to the Cauchy problem for the damped
nonlinear Schr\"{o}dinger equations,}
\newblock SIAM J. Math. Anal.
{\bf 15} (1984), no. 2, 357--366.

\bibitem{MTsutsumi2}{M. Tsutsumi,}
\newblock{\em On global solutions to the initial-boundary value problem for the damped
nonlinear Schr\"{o}dinger equations,}
\newblock J. Math. Anal. Appl.
{\bf 145} (1990), no. 2, 328--341.

\bibitem{YTsutsumi}{Y. Tsutsumi,}
\newblock{\em Scattering problem for nonlinear Schr\"{o}dinger equations,}
\newblock Henri Poincar\'{e} Phys.Th\'{e}or. 
{\bf 43} (1985), 321--347.

\bibitem{Tsutsumi-Yajima}{Y. Tsutsumi, K. Yajima,}
\newblock{\em The asymptotic behavior of nonlinear Schr\"{o}dinger equations,}
\newblock Bull. Amer. Math. Soc.
{\bf 11} (1984), no. 1, 186--188.

\bibitem{Wang}{B.~Wang and H.~Hudzik,}
\newblock{\em The global Cauchy problem for the NLS and NLKG with small rough data,}
\newblock J. Differential Equations
{\bf 232} (2007), no.~1, 36--73.
\end{thebibliography}
\end{document}